\documentclass[preprint,review,10pt]{elsarticle}% Math and Physical Sciences
\usepackage{amsmath,amssymb,amsfonts,amsthm}
\usepackage{hyperref}
\usepackage{graphicx}
\usepackage{mathrsfs}

\usepackage{float}
\usepackage{color}
\usepackage{cases}

\usepackage{pifont}
\usepackage{latexsym}
\usepackage{mathrsfs}
\usepackage{dsfont}
\usepackage{eufrak}
\usepackage{multirow}
\biboptions{sort&compress}
\usepackage{graphicx}
\usepackage{hyperref}
\usepackage{color}
\usepackage{diagbox}

\usepackage{enumerate}
\usepackage{cases}
\usepackage{xcolor}
\theoremstyle{thmstyleone}%
\newtheorem{theorem}{Theorem}
\usepackage[justification=centering]{caption}

\theoremstyle{thmstyletwo}%
\newtheorem{remark}{Remark}%

\theoremstyle{thmstylethree}%

\newtheorem{assumption}{Assumption}
\newtheorem{lemma}{Lemma}
\let\mr=\mathrm

\begin{document}
\begin{frontmatter}

\title{Supercloseness analysis of the nonsymmetric interior penalty Galerkin  method  on Bakhvalov-type mesh\tnoteref{funding} }

\tnotetext[funding]{
This research is supported by National Natural Science Foundation of China (11771257, 11601251), Shandong Provincial Natural Science Foundation, China (ZR2021MA004).
%, Shandong Provincial Natural Science Foundation, China (ZR2021MA004).
}

\author[label1] {Xiaoqi Ma \fnref{cor1}}
\author[label1] {Jin Zhang \corref{cor2}}
\cortext[cor2] {Corresponding email: jinzhangalex@sdnu.edu.cn }
\fntext[cor1] {Email: xiaoqiMa@hotmail.com }
\address[label1]{School of Mathematics and Statistics, Shandong Normal University, Jinan 250014, China}

\begin{abstract}
In this paper, we study the convergence of the nonsymmetric interior penalty Galerkin (NIPG) method on a Bakhvalov-type mesh for the first time. For this purpose, a new composite interpolation is designed, which solves the inherent difficulty of analysis on Bakhvalov-type meshes. More specifically, Gau{\ss} Radau interpolation and Gau{\ss} Lobatto interpolation are used outside and inside the layer, respectively. On the basis of that, by choosing the specific values of the penalty parameters at different mesh points, we derive the supercloseness of $k+\frac{1}{2}$th order ($k\ge 1$), and prove the convergence of optimal order in an energy norm. The theoretical conclusion is consistent with the numerical results.
\end{abstract}

\end{frontmatter}

\section{Introduction}
In recent years, with the wide application of singularly perturbated problems in practical life, relevant numerical methods have attracted the attention of more and more researchers, see \cite{Gov1Moh2:2022-motified, Moh1Nat2:2010-motified, Moh1Red2:2015-E, Zha1Lv2:2021-H, Zha1Lv2:2022-S, Sah1Moh2:2019-P, Fal1Heg2Mil3:2000-R, Kad1Gup2:2010-A, Mad1Sty2:2003-motified,  Zha1Liu2:2020-O} and their references. It is worth noting that the exact solutions of this kind of problem usually change sharply locally, resulting boundary layers or interior layers. 
%{\color{blue}Therefore, the traditional numerical methods do not work well, unless the mesh discretization used is extremely fine.} 
In order to better resolve these layers, researchers designed a simple and effective mesh strategy--layer adapted meshes, whose most representative ones are Bakhvalov-type meshes \cite{Bak1:1969-motified} and Shishkin meshes \cite{Shi1:1990-G}. In the numerical experiment of \cite{Lin1Sty2:2001-N}, we find that even if we use the standard Galerkin method on a layer adapted mesh, there still exist small oscillations. Therefore, it is necessary to consider strong stable numerical methods on layer adapted meshes, such as continuous interior penalty stabilization, the streamline diffusion finite element method, local projection stabilization and discontinuous Galerkin method, see literature \cite{Adj1Kia2:2005-S, Styn1Styn2:2018-Convection-diffusion, Fal1Heg2Mil3:2000-R, Roo1Sty2Tob3:2008-R, Hou1Sch2:2002-D, Sty1Tob2:2003-mptified} for more details.

Since the 1970s, the nonsymmetric interior penalty Galerkin (NIPG) method has gradually become a popular stabilization technique. Because this method applies an interior penalty term to restrain the discontinuity across element boundaries, it has flexibility and advantages that the traditional finite element method does not have. In addition, compared with the incomplete interior penalty Galerkin method and the symmetrical interior penalty Galerkin method, a prominent advantage of the NIPG method is that it has strong stability and has no strict restrictions on the value of penalty parameters. With the continuous research of scholars, the convergence analysis of NIPG method is not uncommon. 
For example, Roos and Zarin analyzed the convergence of a two-dimensional convection diffusion problem of NIPG method on a Shishkin mesh with bilinear elements in \cite{Zar1Roo2:2005-I}. Then, Zarin used the NIPG method on a Shishkin-type mesh, and derived the convergence of almost $k$ order \cite{Roo1Zar2:2007-motified}. In \cite{Zhu1Tan2Yin3:2015-H}, Zhu et al. applied the NIPG method on a Shishkin-type mesh, and proved the method is uniformly convergent in an energy norm. It can be seen that the analysis in the above work only focuses on Shishkin meshes, and the convergence analysis of NIPG method on Bakhvalov-type meshes has not been established.
This is because the convergence of the convection term cannot be analyzed on the element in the vicinity of the transition point near the layer. To deal with this difficulty, we propose a new interpolation, which lays a foundation for the analysis of the NIPG method in two-dimensional case.

In this paper, for a one-dimensional singularly perturbed problem, the supercloseness analysis of NIPG method is studied on a Bakhvalov-type mesh for the first time. For the sake of the desired results, we use Gau{\ss} Radau interpolation outside the layer, and Gau{\ss} Lobatto interpolation inside the layer. Note that the division inside and outside the layer needs to be based on the characteristics of mesh and analysis. Then we derive the penalty parameters at different element boundaries and the supercloseness of $k+\frac{1}{2}$ order. 

The rest of the paper is organized as follows. Firstly, we describe a continuous problem and provide some basic assumptions in Section 2. In addition, the NIPG method on Bakhvalov-type mesh is introduced. Then in Section 3 a new interpolation is defined, further, the corresponding interpolation error estimate is obtained. In Section 4, the uniform supercloseness related to perturbation parameter is presented. At last, we provide some numerical results to verify the main conclusion.

Throughout the paper, let $C$ be a general positive constant, which is independence of the perturbation parameter $\varepsilon$ and the mesh parameter $N$. Furthermore, assume that $k$ is a fixed integer and satisfy the condition $k\ge 1$.
\section{Continuous problem}\label{sec:mesh,method}
We consider the following singularly perturbed problem:
\begin{equation}\label{eq:S-1}
\begin{aligned}
& Lu:=-\varepsilon u''(x)+b(x)u'(x)+c(x)u(x)=f(x) \quad \text{$x \in \Omega:= (0,1)$},\\
& u(0)= u(1)=0,
\end{aligned}
\end{equation}
where $0<\varepsilon \ll 1$, and $b(x)$ is the convection coefficient  satisfying $b(x)\ge\alpha>0$ on $\bar{\Omega}$. Furthermore, for some fixed constant $\gamma$, assume that
\begin{equation}\label{eq:SPP-condition-1}
 c(x)-\frac{1}{2}b'(x)\ge \gamma>0, \quad\forall x\in\bar{\Omega}.
\end{equation}
Here $b$, $c$ and $f$ are sufficiently smooth. Due to $\varepsilon$ can be arbitrarily small, the exact solution $u$ of \eqref{eq:S-1}  typically features a boundary layer near $x= 1$, whose width is $\mathcal{O}(\varepsilon \ln (1/\varepsilon) )$.

Below, we introduce \emph{a priori} information of the solution, which is the basis of our analysis.
\begin{theorem}\label{eq:AASS}
Suppose that $q$ is a positive integer. Assume that \eqref{eq:SPP-condition-1} holds true and $b$, $c$, $f$ are sufficiently smooth. Then the solution $u$ of \eqref{eq:S-1} can be decomposed into $u = S + E$, where the smooth component $S$ and the layer component $E$ satisfy $LS=f$ and $LE=0$, separately. Then for $0\le l\le q$
\begin{equation}\label{eq:decomposition}
\begin{aligned}
&\vert S^{(l)}(x)\vert\le C,\quad
&\vert E^{(l)}(x)\vert\le C\varepsilon^{-l}e^{-\alpha(1-x)/\varepsilon}.
\end{aligned}
\end{equation}
In particular, when $b, c, f\in C^{\infty}(\Omega)$, \eqref{eq:decomposition} holds for any $q\in \mathbb{N}$.
\end{theorem}
\begin{proof}
From \cite{Fal1Heg2Mil3:2000-R}, this conclusion can be obtained directly.
\end{proof}

\subsection{Bakhvalov-type mesh}
First, we suppose that the mesh points $\Omega_{N}=\{x_{j}\in \Omega : j=0, 1, 2,\cdots, N\}$ and a partition of $\Omega$
\begin{equation*}
\mathcal{T}_{N} =\{{I_{j} = [x_{j-1}, x_{j}]: j = 1, 2, \cdots, N}\}.
\end{equation*}
Then assume that $h_{j}=x_{j}-x_{j-1}$ is denoted as the length of $I_{j}$, while $I$ represents the general interval.

In order to better characterize the change of solution in the region $\Omega$, we adopt a Bakhvalov-type mesh, the specific description is as follows: The domain $\bar{\Omega}$ is divided as $\bar{\Omega}= [0, \tau]\cup [\tau, 1]$, where the transition point $\tau=1+\frac{\sigma \varepsilon}{\alpha} \ln \varepsilon$ satisfies $\tau \ge 1/2$ and $\sigma\ge k+1$. Let $N\in \mathbb{N}$ be an integer divisible by $2$ and $N\ge 4$. Each subdomain contains $N/2$ mesh points. Therefore, the mesh generating function is defined as
\begin{equation}\label{eq:Bakhvalov-type mesh-Roos}
x= \psi(t)=
\left\{
\begin{aligned}
&1+\frac{\sigma\varepsilon}{\alpha}\ln [1+2(1-\varepsilon)(t-1)],\quad t\in [\frac{1}{2}, 1]\\
&2\tau t,\quad t\in [0, \frac{1}{2}).
\end{aligned}
\right.
\end{equation}
Obviously, $x_{N/2}=\tau$ can be obtained. 

\begin{assumption}\label{ass:S-1}
In this paper, we will make an assumption that
\begin{equation*}
\varepsilon \le C N^{-1},
\end{equation*}
as is not a restriction in practice.
\end{assumption}
When $\varepsilon \ge CN^{-1}$, the layer in the exact solution of \eqref{eq:S-1} is weak, and it can be well resolved by uniform meshes. And the relevant theoretical analysis can be covered by the usual discontinuous finite element theory \cite{Di1Dan2Ern3-2012:M}. 
\subsection{The NIPG method}
In the following, we present some basic notions. Let $m$ be a nonnegative integer and for $I\in \mathcal{T}_{N}$, the space of order $m$ is denoted as 
\begin{equation*}
H^{m}(\Omega, \mathcal{T}_{N}) = \{\omega\in L^{2}(\Omega): \omega\vert_{I}\in H^{m}(I), \text{for all $I\in \mathcal{T}_{N}$}\}.
\end{equation*}
Then the corresponding norm and seminorm can be defined by
\begin{equation*}
\Vert w\Vert^{2}_{m,\mathcal{T}_{N}}=\sum_{j=1}^{N} \Vert w\Vert^{2}_{m,I_{j}},\quad \vert w\vert^{2}_{m,\mathcal{T}_{N}}= \sum_{j=1}^{N}\vert w\vert^{2}_{m,I_ {j}},
\end{equation*}
where $\Vert\cdot\Vert_{m, I_j}$ is the usual Sobolev norm and $\vert\cdot\vert_{m, I_j}$ is the usual semi-norm in $H^{m}(I_j)$. 
In particular, $\Vert\cdot\Vert_{I}$ and $(\cdot, \cdot)_{I}$ are usually used to stand for the $L^{2}(I)$-norm and the $L^{2}(I)$-inner product, respectively. Then on Bakhvalov-type mesh, we define the finite element space as
\begin{equation*}
V_{N}^{k}=\{v\in L^{2}(\Omega): v\vert_{I}\in \mathbb{P}_{k}(I),\quad \forall I\in \mathcal{T}_{N}\}.
\end{equation*}
Here $\mathbb{P}_{k}(I)$ is the space of polynomials of degree at most $k$ on $I$. It is worth noting that, the functions in $V_{N}^{k}$ are discontinuous at the boundary between two adjacent elements, so they might be multivalued at each node $\{x_{j}\}, j=0, 1, \cdots, N$. For a function $u\in H^{1}(\Omega, \mathcal{T}_{N})$, we define the jump and average at the interior node as
\begin{equation*}
[u(x_{j})]=u(x_{j}^{-})-u(x_{j}^{+}),\quad \{u(x_{j})\}=\frac{1}{2}\left(u(x_{j}^{+})+u(x_{j}^{-})\right),
\end{equation*}
where $u(x_{j}^{+})= \lim\limits_{x\rightarrow x_{j}^{+}}u(x)$ and $u(x_{j}^{-})=\lim\limits_{x\rightarrow x_{j}^{-}}u(x)$ for all $j=1, \cdots, N-1$.
In general, the definitions of jump and average can be extended to the boundary nodes $x_{0}$ and $x_{N}$, that is
\begin{equation*}
[u(x_{0})]=-u(x_{0}^{+}),\quad \{u(x_{0})\}=u(x_{0}^{+}),\quad [u(x_{N})]=u(x_{N}^{-}),\quad \{u(x_{N})\}=u(x_{N}^{-}).
\end{equation*}

Now we present the weak formulation for the problem \eqref{eq:S-1}: Find $u_{N} \in V_{N}^{k}$ such that
\begin{equation}\label{eq:SD}
B(u_{N},v_{N})=L(v_{N}) \quad \text{for all $v_{N} \in V_{N}^{k}$},
\end{equation}
where
\begin{equation*}
\begin{aligned}
&B(u, v)=B_{1}(u, v)+B_{2}(u, v)+B_{3}(u, v),\\
&B_{1}(u, v)=\sum_{j=1}^{N}\int_{I_{j}}\varepsilon u'v'\mr{d}x-\varepsilon \sum_{j=0}^{N}\{u'(x_{j})\}[v(x_{j})]+\varepsilon \sum_{j=0}^{N}[u(x_{j})]\{v'(x_{j})\}+\sum_{j=0}^{N} \mu(x_{j})[u(x_{j})][v(x_{j})],\\
&B_{2}(u, v)=\sum_{j=1}^{N}\int_{I_{j}}b(x) u'v\mr{d}x-\sum_{j=0}^{N-1}b(x_{j})[u(x_{j})]v(x_{j}^{+}),\\
%&B_{2}(u,v)=-\sum_{j=1}^{N}\int_{I_{j}}b'(x) uv\mr{d}x-\sum_{j=1}^{N}\int_{I_{j}}b(x) uv'\mr{d}x-\sum_{i=1}^{N}b(x_{j})[v(x_{j})]u(x_{j}^{-}),\\
&B_{3}(u, v)=\sum_{j=1}^{N}\int_{I_{j}}c(x)uv\mr{d}x,\\
&L(v)=\sum_{j=1}^{N}\int_{I_{j}}f v\mr{d}x.
\end{aligned}
\end{equation*}
Note that the penalty  parameters $\mu(x_{j}) (j=0, 1, \cdots, N)$ associated with $x_{j}$ are some nonnegative constants. In this paper, we will take $\mu(x_{j})$ as
\begin{equation}\label{eq: penalization parameters}
\mu(x_{j})=\left\{
\begin{aligned}
&1,\quad 0\le j\le N/2,\\
&N^{2},\quad N/2+1\le j\le N.
\end{aligned}
\right.
\end{equation}

\begin{lemma}\label{Galerkin orthogonality property}
Let $u$ be the exact solution of \eqref{eq:S-1}, then for all $v\in V_{N}^{k}$, we have the following Galerkin orthogonality
\begin{equation*}
B(u-u_{N}, v)=0,
\end{equation*}
where $B(\cdot,\cdot)$ is defined as \eqref{eq:SD}.
\end{lemma}
\begin{proof}
Using the similar arguments in \cite{Zhu1Tan2Yin3:2015-H}, we draw this conclusion directly.
\end{proof}

For all $v\in V_{N}^{k}$, the natural norm associated with $B(\cdot,\cdot)$ is defined by
\begin{equation}\label{eq:SS-1}
\Vert v \Vert_{NIPG}:=
\left(\varepsilon \sum_{j=1}^{N}\Vert v'\Vert_{I_{j}}^{2}+\sum_{j=1}^{N}\gamma\Vert v \Vert_{I_{j}}^{2}+\sum_{j=0}^{N}\left(\mu(x_{j})+\frac{1}{2}b(x_{j})\right)[v(x_{j})]^{2}\right)^{\frac{1}{2}}.
\end{equation}
According to the similar arguments in \cite{Zhu1Tan2Yin3:2015-H}, it is easy to see that one has the coercivity
\begin{equation}\label{eq:coercity}
B(v_{N}, v_{N}) \ge \Vert v_{N} \Vert_{NIPG}^2\quad \forall v_{N}\in V_{N}^{k}.
\end{equation}
Then from Lax-Milgram lemma \cite[Theorem 1.1.3]{Cia1:2002-motified} and \eqref{eq:coercity}, $u_{N}$ is the unique solution of \eqref{eq:SD}.

\subsection{Some preliminary conclusions}
\begin{lemma}
Suppose that Assumption \ref{ass:S-1} holds true. Then on Bakhvalov-type mesh \eqref{eq:Bakhvalov-type mesh-Roos}, one has
\begin{align}
& h_{N/2+2}\ge \cdots\ge h_{N},\label{eq:mesh-1}\\
&\frac{\sigma\varepsilon}{4\alpha}\le h_{N/2+2}\le \frac{\sigma\varepsilon}{\alpha}, \label{eq:mesh-2}\\
&\frac{\sigma\varepsilon}{2\alpha}\le h_{N/2+1}\le \frac{2\sigma}{\alpha}N^{-1},\label{eq:mesh-3}\\
&N^{-1}\le h_{j}\le 2N^{-1},\quad 1\le j\le N/2.\label{eq:mesh-4}
\end{align}
Furthermore, we present estimates at some special points,
\begin{align}
x_{N/2+1}\le 1-C\frac{\sigma\varepsilon}{\alpha}\ln N, \quad x_{N/2}\le 1+C\frac{\sigma\varepsilon}{\alpha}\vert \ln \varepsilon\vert\label{eq:mesh-5}.
%&x_{N/2+2}\le 1-C\frac{\sigma\varepsilon}{\alpha}\ln N,\label{eq:mesh-6}.
\end{align}
In particular, for $N/2+2\le j\le N$ and $0\le \lambda\le\sigma$, 
\begin{equation}
h_{j}^{\lambda}\max_{x_{j-1}\le x\le x_{j}}e^{-\alpha(1-x)/\varepsilon}\le h_{j}^{\lambda}e^{-\alpha(1-x_{j})/\varepsilon}\le C\varepsilon^{\lambda}N^{-\lambda}.\label{eq:mesh-7}
\end{equation}
\end{lemma}
\begin{proof}
Applying the similar method in \citep[Lemma 3]{Zha1Liu2:2020-O}, we can derive this lemma without any difficulties.
\end{proof}
\begin{lemma}
Suppose that Assumption \ref{ass:S-1} holds, then on Bakhvalov-type mesh \eqref{eq:Bakhvalov-type mesh-Roos}, 
\begin{equation*}
%\vert E(x_{N/2+2})\vert\le CN^{-\sigma},\quad
\vert E(x_{N/2+1})\vert\le CN^{-\sigma},\quad \vert E(x_{N/2})\vert\le C\varepsilon^{\sigma}.
\end{equation*}
\end{lemma}
\begin{proof}
Through \eqref{eq:Bakhvalov-type mesh-Roos} and \eqref{eq:mesh-5}, the conclusion of this lemma can be obtained.
\end{proof}
\section{Interpolation and interpolation error}

\subsection{Interpolation}\label{QQW-1}
Below we introduce a new interpolation operator $\Pi$, that is
\begin{equation}\label{eq:H-1}
(\Pi u)\vert_{I}=\left\{
\begin{aligned}
& (L_{k}u)\vert_I, \quad \text{if $I\subset [x_{N/2+1}, 1]$},\\
& (P_{h}u)\vert_I,\quad \text{if $I\subset [0, x_{N/2+1}]$},
\end{aligned}
\right.
\end{equation}
where $L_{k}u$ is $k$-degree Gau{\ss} Lobatto interpolation, and $P_{h}u$ is Gau{\ss} Radau interpolation of $u$. Then, we shall provide the definitions of these two interpolations.

First, assume that $x_{j-1}= s_{0}< s_{1}< \cdots < s_{k}=x_{j}$ is the Gau{\ss} Lobatto points, where $s_{1}, s_{2}, \cdots, s_{k-1}$ are zeros of the derivative of $k$-degree Legendre polynomial on $I_{j}$. For $\varphi\in C(\overline{\Omega})$, let $(L_{k} \varphi)\vert_{I_j}$ for $j=1, \ldots, N$ be the Lagrange interpolation of degree $k$ at the Gau{\ss} Lobatto points $\{t_{m}\}_{m=0}^{k}$.
Then for $\varphi(x)\in H^{k+2}(I_{j})$, we have
\begin{equation}\label{eq: Gauss-Lobatto-1}
\vert(\varphi'-(L_{k} \varphi)', v')_{I_{j}}\vert\le Ch_{j}^{k+1}\vert \varphi\vert_{k+2, I_{j}}\vert v\vert_{1, I_{j}}\quad\text{for all $v\in \mathbb{P}_{k}$}.
\end{equation}

Furthermore, if $k\ge 1$, we define Gau{\ss} Radau interpolation $P_{h}u\in V_{h}^{k}$ by: For $j=1, 2, \cdots, N$, 
\begin{align}
&\int_{I_j}(P_{h}u)v_{h}\mr{d}x=\int_{I_j}uv_{h}\mr{d}x,\quad \forall v_{h}\in \mathbb{P}_{k-1},\label{eq:J-1}\\
&(P_{h}u)(x_{j}^{-})=u(x_{j}^{-}),\label{eq:J-2}
\end{align}
see \cite{Che1Shu2:2008-S} for more details.
\begin{remark}\label{special}
Now we provide the situation at $x_{N/2+1}$. On the one hand, by using the definition of Gau{\ss} Radau interpolation \eqref{eq:J-2}, there is
\begin{equation}\label{eq: Kk-11}
(u-P_{h} u)(x_{N/2+1}^{-})=0.
\end{equation}
On the other hand, according to Gau{\ss} Lobatto interpolation, 
\begin{equation}\label{eq: Kk-12}
(u-L_{k} u)(x_{N/2+1}^{+})=0.
\end{equation}
Therefore, from \eqref{eq:H-1}, \eqref{eq: Kk-11} and \eqref{eq: Kk-12}, we have
\begin{equation*}
[(u-\Pi u)(x_{N/2+1})]=0.
\end{equation*}
\end{remark}

\subsection{Interpolation error}
Recall that $L_{k}$ is the Lagrange interpolation operator with Gau{\ss} Lobatto points as the interpolation nodes. From the interpolation theories  in Sobolev spaces \citep[Theorem 3.1.4]{Cia1:2002-motified}, for all $v\in W^{k+1,m}(I_j)$,  
\begin{equation}\label{eq:interpolation-theory}
\Vert v-L_{k}v \Vert_{W^{l,n}(I_j)}\le C h_{j}^{k+1-l+1/n-1/m}\vert v \vert_{W^{k+1,m}(I_j)}, 
\end{equation}
where $l=0, 1$ and $1\le m, n\le \infty$.   
Then recall that $P_{h}v$ is Gau{\ss} Radau interpolation of $v$. According to the arguments in \cite{Cia1:2002-motified}, for all $v\in H^{k+1}(I_{j})$, we have
\begin{equation}\label{eq:interpolation-theory-1}
\Vert v-P_{h}v \Vert_{I_{j}}+h_{j}^{\frac{1}{2}}\Vert v-P_{h}v \Vert_{L^{\infty}(I_j)}\le C h_{j}^{k+1}\vert v \vert_{k+1, I_{j}}, \quad j=1, 2, \cdots, N.
\end{equation}
On the basis of that, it is straightforward for us to derive the following error.

\begin{lemma}
Suppose that Assumption \ref{ass:S-1} hold and $\mu(x_{j})$ is presented in \eqref{eq: penalization parameters}. Then on Bakhvalov-type mesh \eqref{eq:Bakhvalov-type mesh-Roos} with $\sigma\ge k+1$, one has
\begin{align}
&\Vert S-P_{h}S\Vert_{[0, x_{N/2+1}]}+\Vert u-\Pi u\Vert_{[0, 1]}\le CN^{-(k+1)},\label{eq:QQ-1}\\
&\Vert (S-\Pi S)'\Vert_{[0,1]}\le C N^{-k},\label{eq: QQ-2}\\
&\Vert (E-P_{h}E)'\Vert_{[0, x_{N/2+1}]}\le C\varepsilon^{\sigma}N+C\varepsilon^{-\frac{1}{2}}N^{-\sigma},\label{eq:QQ-3}\\
&\Vert u-\Pi u\Vert_{L^{\infty}(I_{j})}\le CN^{-(k+1)},\quad j=1, 2, \cdots, N\label{eq:QQ-4}\\
&\Vert L_{k}u-u\Vert_{NIPG, [x_{N/2+2}, 1]}\le CN^{-k},\label{post-process-1}\\
&\Vert u-P_{h}u\Vert_{NIPG, [0, x_{N/2+1}]}\le C N^{-(k+\frac{1}{2})}.\label{post-process-2}
\end{align}
\end{lemma}
\begin{proof}
First, \eqref{eq:interpolation-theory} and \eqref{eq:interpolation-theory-1} yield (\ref{eq:QQ-1}-\ref{eq:QQ-4}) easily, thus, we just estimate \eqref{post-process-1} and \eqref{post-process-2}.

According to the NIPG norm \eqref{eq:SS-1} and the definition of Gau{\ss} Lobatto interpolation, one has $[(L_{k}u-u)(x_{j})]=0, j= N/2+2, \cdots, N$. Further, 
\begin{equation*}
\begin{aligned}
\Vert L_{k}u-u\Vert_{NIPG, [x_{N/2+2}, 1]}^{2}&=\varepsilon \sum_{j=N/2+2}^{N}\Vert (L_{k}u-u)'\Vert_{I_{j}}^{2}+\sum_{j=N/2+2}^{N}\gamma\Vert L_{k}u-u\Vert_{I_{j}}^{2}.
\end{aligned}
\end{equation*}
Then from Theorem \ref{eq:AASS}, we will divide the first item for analysis, that is,
\begin{equation*}
\begin{aligned}
\vert\varepsilon \sum_{j=N/2+2}^{N}\Vert (L_{k}u-u)'\Vert_{I_{j}}^{2}\vert
&\le C\varepsilon \sum_{j=N/2+2}^{N}\Vert (L_{k}S-S)'\Vert_{I_{j}}^{2}+C\varepsilon \sum_{j=N/2+2}^{N}\Vert (L_{k}E-E)'\Vert_{I_{j}}^{2}.
\end{aligned}
\end{equation*}
Actually, from \eqref{eq:decomposition} and \eqref{eq:interpolation-theory}, we can derive
\begin{equation*}
\varepsilon \sum_{j=N/2+2}^{N}\Vert (L_{k}S-S)'\Vert_{I_{j}}^{2}\le C\varepsilon\sum_{j=N/2+2}^{N}h^{2k}_{j}\Vert S^{(k+1)}\Vert_{I_{j}}^{2}\le C\varepsilon^{2k+2} N,
\end{equation*}
where we also use \eqref{eq:mesh-1} and \eqref{eq:mesh-2}. Note that \eqref{eq:decomposition}, \eqref{eq:mesh-7} and \eqref{eq:interpolation-theory} yield
\begin{equation*}
\begin{aligned}
&\varepsilon \sum_{j=N/2+2}^{N}\Vert (L_{k}E-E)'\Vert_{I_{j}}^{2}\le C\varepsilon \sum_{j=N/2+2}^{N}h^{2k}_{j}\Vert E^{(k+1)}\Vert_{I_{j}}^{2}\\
&\le C\varepsilon \sum_{j=N/2+2}^{N}h^{2k+1}_{j}e^{-2\alpha(1-x_{j})/\varepsilon}\varepsilon^{-2(k+1)}\\
&\le C\sum_{j=N/2+2}^{N}N^{-(2k+1)}\le CN^{-2k}.
\end{aligned}
\end{equation*}
Moreover, by using the similar method, one has
\begin{equation*}
\vert\sum_{j=N/2+2}^{N}\gamma \Vert L_{k}u-u\Vert_{I_{j}}^{2}\vert\le C\sum_{j=N/2+2}^{N}\Vert L_{k}S-S\Vert_{I_{j}}^{2}+C\sum_{j=N/2+2}^{N}\Vert L_{k}E-E\Vert_{I_{j}}^{2}\le C\varepsilon N^{-(2k+1)}.
\end{equation*}
Thus the derivation of \eqref{post-process-1} has been completed.

Now let's analyze the estimate of \eqref{post-process-2}. Through \eqref{eq:SS-1} and Remark \ref{special},
\begin{equation*}
\begin{aligned}
\Vert u-P_{h}u\Vert_{NIPG, [0,x_{N/2+1}]}^{2}&=\varepsilon \sum_{j=1}^{N/2+1}\Vert (u-P_{h}u)'\Vert_{I_{j}}^{2}+\sum_{j=1}^{N/2+1}\gamma\Vert u-P_{h}u\Vert_{I_{j}}^{2}\\&+\sum_{j=0}^{N/2}\mu(x_{j})[(u-P_{h}u)(x_{j})]^{2}+\frac{1}{2}\sum_{j=0}^{N/2}b(x_{j})[(u-P_{h}u)(x_{j})]^{2}\\&=\Lambda_{1}+\Lambda_{2}+\Lambda_{3}+\Lambda_{4}.
\end{aligned}
\end{equation*}
Next, we estimate $\Lambda_{1}$, $\Lambda_{2}$, $\Lambda_{3}$ and $\Lambda_{4}$ in turn.

For $\Lambda_{1}$, we first decompose it into the following forms,
\begin{equation*}
\Lambda_{1}\le \varepsilon\sum_{j=1}^{N/2+1}\Vert(S-P_{h}S)'\Vert_{I_{j}}^{2}+\varepsilon\sum_{j=1}^{N/2+1}\Vert(E-P_{h}E)'\Vert_{I_{j}}^{2}.
\end{equation*}
For one thing, through \eqref{eq:interpolation-theory-1} and \eqref{eq:decomposition}, 
\begin{equation*}
\vert\varepsilon\sum_{j=1}^{N/2+1}\Vert(S-P_{h}S)'\Vert_{I_{j}}^{2}\vert\le C\varepsilon \sum_{j=1}^{N/2+1}h_{j}^{2k}\Vert S^{(k+1)}\Vert^{2}_{I_{j}}\le C\varepsilon N^{-2k}.
\end{equation*}
For another, by means of the triangle inequality and the inverse inequality \citep[ Theorem 3.2.6]{Cia1:2002-motified}, we have
\begin{equation*}
\begin{aligned}
\vert\varepsilon\sum_{j=1}^{N/2+1}\Vert(E-P_{h}E)'\Vert_{I_{j}}^{2}\vert\le C\varepsilon^{2\sigma+1}N^{2}+CN^{-2\sigma},
\end{aligned}
\end{equation*}
where \eqref{eq:mesh-3} and \eqref{eq:mesh-4} have been used.

For $\Lambda_{2}$, the triangle inequality and \eqref{eq:interpolation-theory-1} can yield
\begin{equation*}
\vert\sum_{j=1}^{N/2+1} \gamma \Vert u-P_{h}u\Vert_{I_{j}}^{2}\vert\le C\sum_{j=1}^{N/2+1}\Vert S-P_{h}S\Vert_{I_{j}}^{2}+C\sum_{j=1}^{N/2+1} \Vert E-P_{h}E\Vert_{I_{j}}^{2}\le CN^{-2(k+1)}.
\end{equation*}

Recall that $\mu(x_{j})$ is defined as \eqref{eq: penalization parameters}, then one has
\begin{equation*}
\begin{aligned}
&\vert\sum_{j=1}^{N/2}\mu(x_{j})[(S-P_{h}S)(x_{j})]^{2}\vert\le C\sum_{j=1}^{N/2}\mu(x_{j})\Vert S-P_{h}S\Vert_{L^{\infty}(I_{j}\cup I_{j+1})}^{2}\le CN^{-(2k+1)},\\
&\vert\sum_{j=1}^{N/2}\mu(x_{j})[(E-P_{h}E)(x_{j})]^{2}\vert\le C\sum_{j=1}^{N/2}\mu(x_{j})\Vert E-P_{h}E\Vert_{L^{\infty}(I_{j}\cup I_{j+1})}^{2}\le C\varepsilon^{2\sigma}N.
\end{aligned}
\end{equation*}
Similarly, it is easy to get
\begin{equation*}
\Lambda_{4}\le CN^{-(2k+1)}.
\end{equation*}
Here $b(x)$ is a smooth function on $[0, 1]$, thus it is bounded. Finally, by some simple calculations, \eqref{post-process-2} can be obtained.
\end{proof}
\begin{theorem}\label{QQQQ-1}
Suppose that Assumption \ref{ass:S-1} and $\mu(x_{i})$ is defined as \eqref{eq: penalization parameters}. On Bakhvalov-type mesh \eqref{eq:Bakhvalov-type mesh-Roos} with $\sigma\ge k+1$, 
\begin{equation*}
\Vert u-\Pi u\Vert_{NIPG}\le CN^{-k},
\end{equation*}
where $u$ is the solution of \eqref{eq:S-1}, and $\Pi u$ is the interpolation defined in \eqref{eq:H-1}.
\end{theorem}
\begin{proof}
From \eqref{post-process-1} and \eqref{post-process-2}, this theorem can be derived easily.
\end{proof}

\begin{lemma}\label{trace inequality}
Suppose that $z\in H^{1}(I_{j}),\quad j=1, 2, \cdots, N$, then 
\begin{equation*}
\vert z(x_{s})\vert^{2}\le 2\left(h_{j}^{-1}\Vert z\Vert_{I_{j}}^{2}+\Vert z\Vert_{I_{j}}\Vert z'\Vert_{I_{j}}\right),\quad s\in\{j-1, j\}.
\end{equation*}
\end{lemma}
\begin{proof}
The corresponding arguments can be found in \cite{Zhu1Tan2Yin3:2015-H}.
\end{proof}

In particular, we simplify $(\Pi u-u)(x)$ to $\eta(x)$ in the following.
\begin{lemma}
Assume that $\varepsilon\le CN^{-1}$ and on the mesh \eqref{eq:Bakhvalov-type mesh-Roos} with $\sigma\ge k+1$, there is
\begin{equation}\label{eq:KK-1}
\{\eta'(x_{j})\}^{2}\le
\left\{
\begin{aligned}
&CN^{-2k}+C\varepsilon^{2\sigma-2}+C\varepsilon^{\sigma-\frac{3}{2}}N^{-(k+\frac{1}{2})},\quad 0\le j\le N/2-1,\\
&C\varepsilon^{-1}N^{-(2k+1)}+C\varepsilon^{-2}N^{-2\sigma}+C\varepsilon^{-\frac{3}{2}}N^{-(\sigma+k+\frac{1}{2})},\quad j=N/2,\\
&C\varepsilon^{-2}N^{-2k},\quad N/2+1\le j\le N.
\end{aligned}
\right.
\end{equation}
\end{lemma}
\begin{proof}
Through \cite{Zhu1Tan2Yin3:2015-H}, we can draw this conclusion.
By means of the definition of average and Lemma \ref{trace inequality},
\begin{equation*}
\begin{aligned}
\{\eta'(x_{j})\}^{2}&=\frac{1}{4}\left(\eta'(x^{-}_{j})+\eta'(x^{+}_{j})\right)^{2}\le \frac{1}{2}\left(\eta'(x^{-}_{j})^{2}+\eta'(x^{+}_{j})^{2}\right)\\
&\le h_{j}^{-1}\Vert \eta' \Vert_{I_{j}}^{2}+\Vert \eta' \Vert_{I_{j}}\Vert \eta'' \Vert_{I_{j}}+h_{j+1}^{-1}\Vert \eta' \Vert_{I_{j+1}}^{2}+\Vert \eta' \Vert_{I_{j+1}}\Vert \eta'' \Vert_{I_{j+1}}.
\end{aligned}
\end{equation*}
In the following, we will estimate $\Vert \eta' \Vert_{I_{j}}$ and $\Vert \eta''\Vert_{I_{j}}$, respectively.

For $1\le j\le N/2$, one has the following estimate,
\begin{equation}\label{eq: FF-1}
\begin{aligned}
\Vert \eta'\Vert_{I_{j}}^{2}&\le \Vert (S-P_{h} S)'\Vert^{2}_{I_{j}}+\Vert (E-P_{h} E)'\Vert_{I_{j}}^{2}\\
&\le Ch_{j}^{2k}\Vert S^{(k+1)}\Vert^{2}_{I_{j}}+\Vert E'\Vert^{2}_{I_{j}}+\Vert (P_{h} E)'\Vert^{2}_{I_{j}}\\
&\le Ch_{j}^{2k+1}\Vert S^{(k+1)}\Vert_{L^{\infty}(I_{j})}^{2}+C\varepsilon^{-2}\int_{I_{j}}e^{-2\alpha(1-x)/\varepsilon}\mr{d}x+Ch_{j}^{-2}\Vert P_{h} E\Vert_{I_{j}}^{2}\\
&\le CN^{-(2k+1)}+C\varepsilon^{-1}\varepsilon^{2\sigma}+Ch_{j}^{-1}\Vert E\Vert^{2}_{L^{\infty}(I_{j})}\\
&\le CN^{-(2k+1)}+C\varepsilon^{-1}\varepsilon^{2\sigma}+CN \varepsilon^{2\sigma}\\
&\le CN^{-(2k+1)}+C\varepsilon^{2\sigma-1},
\end{aligned}
\end{equation}
where \eqref{eq:mesh-4}, \eqref{eq:interpolation-theory-1} and the inverse inequality have been used. In a similar way, we can obtain 
\begin{equation}\label{eq: FF-2}
\Vert \eta''\Vert_{I_{j}}^{2}\le \Vert (S-P_{h} S)''\Vert^{2}_{I_{j}}+\Vert (E-P_{h} E)''\Vert_{I_{j}}^{2}\le CN^{-(2k-1)}+C\varepsilon^{2\sigma-3}.
\end{equation}
To sum up, for $j= 1, \cdots, N/2-1$, there is
\begin{equation*}
\begin{aligned}
\{\eta'(x_{j})\}^{2}&\le h_{j}^{-1}\Vert \eta' \Vert_{I_{j}}^{2}+\Vert \eta' \Vert_{I_{j}}\Vert \eta'' \Vert_{I_{j}}+h_{j+1}^{-1}\Vert \eta' \Vert_{I_{j+1}}^{2}+\Vert \eta' \Vert_{I_{j+1}}\Vert \eta'' \Vert_{I_{j+1}}\\
&\le CN^{-2k}+C\varepsilon^{2(\sigma-1)}+C\varepsilon^{\sigma-\frac{3}{2}}N^{-(k+\frac{1}{2})},
\end{aligned}
\end{equation*}
where note that $\varepsilon\le CN^{-1}$ and $\sigma\ge k+1$.

In addition, when $j=N/2+2, \cdots, N$, for one thing, from \eqref{eq:decomposition}, \eqref{eq:mesh-1}, \eqref{eq:mesh-2} and \eqref{eq:interpolation-theory},
\begin{equation*}
\Vert \eta'\Vert_{I_{j}}^{2}\le \Vert (S-L_{k} S)'\Vert^{2}_{I_{j}}+\Vert (E-L_{k} E)'\Vert_{I_{j}}^{2}\le C\varepsilon^{-1}N^{-(2k+1)}.
\end{equation*}
For another, by means of \eqref{eq:decomposition} and \eqref{eq:interpolation-theory}, one has
\begin{equation*}
\Vert \eta''\Vert_{I_{j}}^{2}\le \Vert (S-L_{k} S)''\Vert^{2}_{I_{j}}+\Vert (E-L_{k} E)''\Vert_{I_{j}}^{2}\le C\varepsilon^{-3}N^{-(2k-1)}.
\end{equation*}
Then for $j=N/2+2, \cdots, N$ we derive 
\begin{equation}\label{eq: FF-3}
\{\eta'(x_{j})\}^{2}\le C\varepsilon^{-2}N^{-2k}.
\end{equation}

Finally, we will analyze the situation when $j=N/2$ and $j=N/2+1$, respectively: On the one hand, when $j=N/2$, 
\begin{equation*}
\begin{aligned}
\{\eta'(x_{N/2})\}^{2}&\le h_{N/2}^{-1}\Vert \eta' \Vert_{I_{N/2}}^{2}+\Vert \eta' \Vert_{I_{N/2}}\Vert \eta'' \Vert_{I_{N/2}}\\
&+h_{N/2+1}^{-1}\Vert \eta' \Vert_{I_{N/2+1}}^{2}+\Vert \eta' \Vert_{I_{N/2+1}}\Vert \eta'' \Vert_{I_{N/2+1}}.
\end{aligned}
\end{equation*}
\eqref{eq:decomposition}, \eqref{eq:mesh-3}, \eqref{eq:interpolation-theory-1} and the inverse inequality \citep[ Theorem 3.2.6]{Cia1:2002-motified} yield
\begin{equation}\label{eq:FF-4}
\begin{aligned}
\Vert\eta'\Vert_{I_{N/2+1}}^{2}&\le \Vert(S-P_{h}S)'\Vert_{I_{N/2+1}}^{2}+\Vert(E-P_{h}E)'\Vert_{I_{N/2+1}}^{2}\\
&\le Ch_{N/2+1}^{2k}\Vert S^{(k+1)}\Vert_{I_{N/2+1}}^{2}+\Vert E'\Vert_{I_{N/2+1}}^{2}+\Vert (P_{h}E)'\Vert_{I_{N/2+1}}^{2}\\
&\le Ch_{N/2+1}^{2k+1}\Vert S^{(k+1)}\Vert_{L^{\infty}(I_{N/2+1})}^{2}+C\varepsilon^{-1}N^{-2\sigma}+Ch_{N/2+1}^{-1}\Vert E\Vert_{L^{\infty}(I_{N/2+1})}^{2}\\
&\le CN^{-2(k+\frac{1}{2})}+C\varepsilon^{-1}N^{-2\sigma}+C\varepsilon^{-1}N^{-2\sigma}\\
&\le CN^{-(2k+1)}+C\varepsilon^{-1}N^{-2\sigma}.
\end{aligned}
\end{equation}
In a same way, we have
\begin{equation}\label{eq:FF-5}
\Vert\eta''\Vert_{I_{N/2+1}}^{2}\le CN^{-(2k-1)}+C\varepsilon^{-3}N^{-2\sigma}.
\end{equation}
Then using \eqref{eq: FF-1}, \eqref{eq: FF-2}, \eqref{eq:FF-4} and \eqref{eq:FF-5}, it is straightforward to derive 
\begin{equation*}
\{\eta(x_{N/2})\}^{2}\le C\varepsilon^{-1}N^{-(2k+1)}+C\varepsilon^{-2}N^{-2\sigma}+C\varepsilon^{-\frac{3}{2}}N^{-(\sigma+k+\frac{1}{2})}.
\end{equation*}
On the other hand, when $j=N/2+1$,
from \eqref{eq: FF-3}, \eqref{eq:FF-4} and \eqref{eq:FF-5}, there is
\begin{equation*}
\{\eta'(x_{N/2+1})\}^{2}\le C\varepsilon^{-2}N^{-2k}.
\end{equation*}
So far, we have proved this conclusion.
\end{proof}

\section{Supercloseness}
Now introduce $\xi:=\Pi u-u_{N}$ and recall $\eta: =\Pi u-u$.
According to \eqref{eq:coercity} and the Galerkin orthogonality, 
\begin{equation}\label{eq:uniform-convergence-1}
\begin{split}
&\Vert \xi \Vert_{NIPG}^2 \le B(\xi,\xi)=B(\Pi u-u+u-u_{N},\xi) =B(\eta,\xi)\\
& =\sum_{j=1}^{N}\int_{I_{j}}\varepsilon \eta'\xi'\mr{d}x+\varepsilon \sum_{j=0}^{N}[\eta(x_{j})]\{\xi'(x_{j})\}-\varepsilon \sum_{j=0}^{N}\{\eta'(x_{j})\}[\xi(x_{j})]\\
&+\sum_{j=0}^{N}\mu(x_{j})[\eta(x_{j})][\xi(x_{j})]+\sum_{j=1}^{N}\int_{I_{j}}b(x)\eta'\xi\mr{d}x-\sum_{j=0}^{N-1}b(x_{j})[\eta(x_{j})]\xi(x_{j}^{+})\\
&+ \sum_{j=1}^{N}\int_{I_{j}}c(x)\eta\xi\mr{d}x\\
&=\sum_{j=1}^{N}\int_{I_{j}}\varepsilon \eta'\xi'\mr{d}x+\varepsilon \sum_{j=0}^{N}[\eta(x_{j})]\{\xi'(x_{j})\}-\varepsilon \sum_{j=0}^{N}\{\eta'(x_{j})\}[\xi(x_{j})]\\
&+\sum_{j=0}^{N}\mu(x_{j})[\eta(x_{j})][\xi(x_{j})]-\sum_{j=1}^{N}\int_{I_{j}}b(x) \eta\xi'\mr{d}x+\sum_{j=1}^{N}b(x_{j})[\xi(x_{j})]\eta(x_{j}^{-})\\
&+\sum_{j=1}^{N}\int_{I_{j}}\left(c(x)-b'(x)\right)\eta\xi\mr{d}x\\
& =:\Gamma_{1}+\Gamma_{2}+\Gamma_{3}+\Gamma_{4}+\Gamma_{5}+\Gamma_{6}+\Gamma_{7}.
\end{split}
\end{equation}
In the sequel, the terms on the right-hand side of \eqref{eq:uniform-convergence-1} are estimated. Firstly, to facilitate analysis, $\Gamma_{1}$ is decomposed as
\begin{equation*}
\Gamma_{1}=\sum_{j=1}^{N/2+1}\int_{I_{j}}\varepsilon \eta'\xi'\mr{d}x+\sum_{j=N/2+2}^{N}\int_{I_{j}}\varepsilon \eta'\xi'\mr{d}x.
\end{equation*}
For $1\le j\le N/2+1$, according to \eqref{eq: QQ-2}, H\"{o}lder inequalities and \eqref{eq:QQ-3}, there is
\begin{equation}\label{eq:convergence-1-1}
\begin{aligned}
\vert\sum_{j=1}^{N/2+1}\int_{I_{j}}\varepsilon \eta'\xi'\mr{d}x\vert&\le C\left(\sum_{j=1}^{N/2+1}\varepsilon\Vert\eta'\Vert^{2}_{I_{j}}\right)^{\frac{1}{2}}\left(\sum_{j=1}^{N/2+1}\varepsilon\Vert\xi'\Vert^{2}_{I_{j}}\right)^{\frac{1}{2}}\\
&\le C\left(\varepsilon^{\frac{1}{2}}N^{-k}+N^{-\sigma}\right)\Vert\xi\Vert_{NIPG}.
\end{aligned}
\end{equation}
In addition, for $j=N/2+2, \cdots, N$, from H\"{o}lder inequalities and \eqref{eq: QQ-2}, 
\begin{equation}\label{eq:convergence-1-2}
\vert\sum_{j=N/2+2}^{N}\int_{I_{j}}\varepsilon (S-L_{k}S)'\xi'\mr{d}x\vert
\le C\varepsilon^{\frac{1}{2}}N^{-k}\Vert\xi\Vert_{NIPG}.
\end{equation}
Through \eqref{eq:mesh-1}, \eqref{eq:mesh-2}, \eqref{eq:mesh-7}, \eqref{eq: Gauss-Lobatto-1} and note $\sigma\ge k+1$, some direct calculations show that
\begin{equation}\label{eq:convergence-1-3}
\begin{aligned}
&\vert\sum_{j=N/2+2}^{N}\int_{I_{j}}\varepsilon(E-L_{k} E)'\xi'\mr{d}x\vert\\
&\le C\sum_{j=N/2+2}^{N}\varepsilon h_{j}^{k+1}\Vert E^{(k+2)}\Vert_{I_{j}}\Vert\xi'\Vert_{I_{j}}\\
&\le C\left(\sum_{j=N/2+2}^{N}\varepsilon h_{j}h_{j}^{2(k+1)}\Vert E^{(k+2)}\Vert^{2}_{L^{\infty}(I_{j})}\right)^{\frac{1}{2}}\left(\sum_{j=N/2+2}^{N}\varepsilon\Vert\xi'\Vert^{2}_{I_{j}}\right)^{\frac{1}{2}}\\
&\le C\left(\sum_{j=N/2+2}^{N}\varepsilon h_{j}h_{j}^{2(k+1)}e^{-2\alpha(1-x_{j})/\varepsilon}\varepsilon^{-2(k+2)}\right)^{\frac{1}{2}}\Vert\xi\Vert_{NIPG}\\
&\le \left(C\sum_{j=N/2+2}^{N}\varepsilon^{-1} h_{j}N^{-2(k+1)}\right)^{\frac{1}{2}}\Vert\xi\Vert_{NIPG}\\
&\le CN^{-(k+\frac{1}{2})}\Vert\xi\Vert_{NIPG}.
\end{aligned}
\end{equation}
From \eqref{eq:convergence-1-1}, \eqref{eq:convergence-1-2}, \eqref{eq:convergence-1-3}, $\varepsilon\le CN^{-1}$ and $\sigma\ge k+1$, we have the following estimate,
\begin{equation}\label{eq:convergence-1}
\Gamma_{1}\le CN^{-(k+\frac{1}{2})}\Vert\xi\Vert_{NIPG}.
\end{equation}

For $\Gamma_{2}$, by using \eqref{eq:KK-1} and the definition of $\mu(x_{j})$, one has
\begin{equation}\label{convergence-2}
\begin{aligned}
\vert\Gamma_{2}\vert
&=\vert\varepsilon \sum_{j=0}^{N}\{\eta'(x_{j})\}[\xi(x_{j})]\vert\\
&\le \left(\sum_{j=0}^{N}\frac{\varepsilon^{2}}{\mu(x_{j})}\{\eta'(x_{j})\}^{2}\right)^{\frac{1}{2}}\left(\sum_{j=0}^{N}\mu(x_{j})[\xi(x_{j})]^{2}\right)^{\frac{1}{2}}\\
&\le \left(\sum_{j=0}^{N/2-1}\frac{\varepsilon^{2}}{\mu(x_{j})}\{\eta'(x_{j})\}^{2}+\frac{\varepsilon^{2}}{\mu(x_{N/2})}\{\eta'(x_{N/2})\}^{2}+\sum_{i=N/2+1}^{N}\frac{\varepsilon^{2}}{\mu(x_{j})}\{\eta'(x_{j})\}^{2}\right)^{\frac{1}{2}}\Vert \xi\Vert_{NIPG}\\
&\le CN^{-(k+\frac{1}{2})}\Vert \xi\Vert_{NIPG}.
\end{aligned}
\end{equation}

Now we decompose $\Gamma_{3}$ into two parts, that is
\begin{equation*}
\Gamma_{3}=-\varepsilon \sum_{j=0}^{N/2}\{\xi'(x_{j})\}[\eta(x_{j})]-\varepsilon \sum_{j=N/2+1}^{N}\{\xi'(x_{j})\}[\eta(x_{j})].
\end{equation*}
According to \eqref{eq:H-1}, the definitions of Gau{\ss} Lobatto interpolation and Remark \ref{special}, we just estimate $-\varepsilon \sum_{j=0}^{N/2}\{\xi'(x_{j})\}[\eta(x_{j})]$. From the inverse inequality and \eqref{eq:QQ-4},
\begin{equation}\label{convergence-3}
\begin{aligned}
\vert-\varepsilon \sum_{j=0}^{N/2}\{\xi'(x_{j})\}[\eta(x_{j})]\vert&\le \vert\varepsilon \{\xi'(x_{0})\}[\eta(x_{0})]\vert +\vert\varepsilon \sum_{j=1}^{N/2}\{\xi'(x_{j})\}[\eta(x_{j})]\vert\\
%&\le C\varepsilon\Vert \Omega\Vert_{L^{\infty}(I_{1})}\Vert \xi'\Vert_{L^{\infty}(I_{1})}+C\left(\varepsilon^{\frac{1}{2}}N^{-k}+N^{-(k+1)}\right)\Vert\xi\Vert_{NIPG}\\
%&\le C\varepsilon\Vert \Omega\Vert_{L^{\infty}(I_{1})}N^{\frac{1}{2}}\Vert \xi'\Vert_{I_{1}}+C\left(\varepsilon^{\frac{1}{2}}N^{-k}+N^{-(k+1)}\right)\Vert\xi\Vert_{NIPG}\\
%&\le C\left(\varepsilon^{\frac{1}{2}}N^{-(k+\frac{1}{2})}+\varepsilon^{\frac{1}{2}} N^{-k}+N^{-(k+1)}\right)\Vert\xi\Vert_{NIPG}\\
&\le \left(C\varepsilon^{\frac{1}{2}}N^{-k}+CN^{-(k+1)}\right)\Vert\xi\Vert_{NIPG},
\end{aligned}
\end{equation}
where the following estimate holds. More specifically,
\begin{equation*}
\begin{aligned}
&\vert\varepsilon \sum_{j=1}^{N/2}[\eta(x_{j})]\{\xi'(x_{j})\}\vert\\
&\le \vert\varepsilon \sum_{j=1}^{N/2-1}[\eta(x_{j})]\{\xi'(x_{j})\}\vert+\vert\varepsilon [\eta(x_{N/2})]\{\xi'(x_{N/2})\}\vert\\
&\le C\varepsilon \sum_{j=1}^{N/2-1}\Vert \eta\Vert_{L^{\infty}(I_{j}\cup I_{j+1})}\Vert \xi'\Vert_{L^{\infty}(I_{j}\cup I_{j+1})}+C\varepsilon\Vert \eta\Vert_{L^{\infty}(I_{N/2}\cup I_{N/2+1})}\Vert \xi'\Vert_{L^{\infty}(I_{N/2}\cup I_{N/2+1})}\\
&\le C\varepsilon \Vert\eta\Vert_{L^{\infty}(I_{j}\cup I_{j+1})}N^{\frac{1}{2}}\sum_{j=1}^{N/2-1}\Vert\xi'\Vert_{I_{j}\cup I_{j+1}}+C\varepsilon\Vert \eta\Vert_{L^{\infty}(I_{N/2}\cup I_{N/2+1})}\varepsilon^{-\frac{1}{2}}\Vert \xi'\Vert_{I_{N/2}\cup I_{N/2+1}}\\
&\le C\left(\varepsilon^{\frac{1}{2}}N^{-k}+N^{-(k+1)}\right)\Vert\xi\Vert_{NIPG},
\end{aligned}
\end{equation*}
where \eqref{eq:mesh-3} and \eqref{eq:mesh-4} have been used.

Then, divide $\Gamma_{4}$ into the following two parts:
\begin{equation*}
\sum_{j=0}^{N}\mu(x_{j})[\eta(x_{j})][\xi(x_{j})]=\sum_{j=0}^{N/2}\mu(x_{j})[\eta(x_{j})][\xi(x_{j})]+\sum_{j=N/2+1}^{N}\mu(x_{j})[\eta(x_{j})][\xi(x_{j})].
\end{equation*} 
According to \eqref{eq:H-1} and Remark \ref{special}, there is $[\eta(x_{j})]=0, j= N/2+1, \cdots, N$. That is to say, we just analyze the first item. Then from \eqref{eq:QQ-4}, 
\begin{equation}\label{convergence-4}
\begin{aligned}
&\vert\sum_{j=0}^{N/2}\mu(x_{j})[\eta(x_{j})][\xi(x_{j})]\vert\\
&\le \left(\sum_{j=0}^{N/2}\mu(x_{j})[\eta(x_{j})]^{2}\right)^{\frac{1}{2}}\left(\sum_{j=0}^{N/2}\mu(x_{j})[\xi(x_{j})]^{2}\right)^{\frac{1}{2}}\\
&\le C \left(\mu(x_{0})\Vert\eta\Vert_{L^{\infty}(I_{1})}^{2}+\sum_{j=1}^{N/2}\mu(x_{j})\Vert \eta\Vert_{L^{\infty}(I_{j}\cup I_{j+1})}^{2}\right)^{\frac{1}{2}}\Vert\xi\Vert_{NIPG}\\
&\le C\left(N^{-2(k+1)}+ N^{-(2k+1)}\right)^{\frac{1}{2}}\Vert \xi\Vert_{NIPG}\\
&\le CN^{-(k+\frac{1}{2})}\Vert\xi\Vert_{NIPG}.
\end{aligned}
\end{equation}

Now we analyze $\Gamma_{5}$ and $\Gamma_{6}$, which are also divided into two parts $1\le j\le N/2+1$ and $N/2+2\le j\le N$. For $1\le j\le N/2+1$, through Remark \ref{special}, \eqref{eq:J-1}, \eqref{eq:J-2}, and assume that $b(x_{j-\frac{1}{2}})$ is the value of $b(x)$ at the midpoint $x_{j-\frac{1}{2}}$ in the interval $I_{j}$, then 
\begin{equation*}
\begin{aligned}
&-\sum_{j=1}^{N/2+1}\int_{I_{j}}b(x) \eta\xi'\mr{d}x-\sum_{j=1}^{N/2+1}b(x_{j})[\xi(x_{j})]\eta(x_{j}^{-})\\
&=-\sum_{j=1}^{N/2+1}\int_{I_{j}}\left(b(x)-b(x_{j-\frac{1}{2}})\right)\eta\xi'\mr{d}x-\sum_{j=1}^{N/2+1}\int_{I_{j}}b(x_{j-\frac{1}{2}})\eta\xi'\mr{d}x-\sum_{j=1}^{N/2+1}b(x_{j})\eta(x_{j}^{-})[\xi(x_{j})]\\
&=-\sum_{j=1}^{N/2+1}\int_{I_{j}}\left(b(x)-b(x_{j-\frac{1}{2}})\right)\eta\xi'\mr{d}x.
\end{aligned}
\end{equation*}
According to the mean value theorem, there is $\xi$ between $x_{j-\frac{1}{2}}$ and $x$ to satisfy
\begin{equation*}
b(x)-b(x_{j-\frac{1}{2}})=b'(\xi)(x-x_{j-\frac{1}{2}}).
\end{equation*}
Note that in this paper $b(x)$ is a smooth function. From the inverse inequality, \eqref{eq:QQ-4} and the Cauchy Schwartz inequality, we have
\begin{equation*}
\begin{aligned}
&\vert-\sum_{j=1}^{N/2+1}\int_{I_{j}}\left(b(x)-b(x_{j-\frac{1}{2}})\right)\eta\xi'\mr{d}x\vert=\vert-\sum_{j=1}^{N/2+1}\int_{I_{j}}b'(\xi)(x-x_{j-\frac{1}{2}})\eta\xi'\mr{d}x\vert\\
&\le C\sum_{j=1}^{N/2+1}h_{j}\Vert \eta\Vert_{L^{\infty}(I_{j})}\Vert\xi'\Vert_{L^{1}(I_{j})}\le C\sum_{j=1}^{N/2+1}h_{j}\Vert \eta\Vert_{L^{\infty}(I_{j})}h_{j}^{-\frac{1}{2}}\Vert\xi\Vert_{I_{j}}\\
&\le C\sum_{j=1}^{N/2+1}N^{-\frac{1}{2}}\Vert\eta\Vert_{L^{\infty}(I_{j})}\Vert\xi\Vert_{I_{j}}\\
&\le CN^{-(k+\frac{3}{2})}\left(\sum_{j=1}^{N/2+1}1^{2}\right)^{\frac{1}{2}}\left(\sum_{j=1}^{N/2+1}\Vert\xi\Vert^{2}_{I_{j}}\right)^{\frac{1}{2}}\\
&\le CN^{-(k+1)}\Vert \xi\Vert_{NIPG}.
\end{aligned}
\end{equation*}
%In summary, for $1\le i\le N/2$, we obtain
%\begin{equation*}
%\mr{V}+\mr{VI}\le CN^{-(k+1)}\Vert\xi\Vert_{NIPG}.
%\end{equation*}

For $N/2+2\le j\le N$, we need to consider the following formula, 
\begin{equation*}
\begin{aligned}
&-\sum_{j=N/2+2}^{N}\int_{I_{j}}b(x) \eta\xi'\mr{d}x-\sum_{j=N/2+2}^{N}b(x_{j})\eta(x_{j}^{-})[\xi(x_{j})]\\
&=-\sum_{j=N/2+2}^{N}\int_{I_{j}}b(x) (S-L_{k}S)\xi'\mr{d}x-\sum_{j=N/2+2}^{N}\int_{I_{j}}b(x) (E-L_{k}E)\xi'\mr{d}x\\&-\sum_{j=N/2+2}^{N}b(x_{j})\eta(x_{j}^{-})[\xi(x_{j})].
\end{aligned}
\end{equation*}
From the H\"{o}lder inequality and the inverse inequality, 
\begin{equation*}
\begin{aligned}
&\vert-\sum_{j=N/2+2}^{N}\int_{I_{j}}b(x)(S-L_{k}S)\xi'\mr{d}x\vert\le C\sum_{j=N/2+2}^{N}\Vert S-L_{k}S\Vert_{L^{\infty}(I_{j})}\Vert\xi'\Vert_{L^{1}(I_{j})}\\
&\le C\Vert S-L_{k}S\Vert_{L^{\infty}(I_{j})}\sum_{j=N/2+2}^{N}h_{j}^{\frac{1}{2}}\Vert\xi'\Vert_{I_{j}}\\
&\le C\varepsilon^{k+1}\left(\sum_{j=N/2+2}^{N} 1^{2}\right)^{\frac{1}{2}}\left(\sum_{j=N/2+2}^{N}\varepsilon\Vert\xi'\Vert_{I_{j}}^{2}\right)^{\frac{1}{2}}\\
%&\le C (N^{-1}\ln N)^{k+\frac{3}{2}} N^{\frac{1}{2}}\Vert\xi\Vert_{NIPG}\\
&\le C\varepsilon^{k+1}N^{\frac{1}{2}}\Vert\xi\Vert_{NIPG}.
\end{aligned}
\end{equation*}

Recall $\sigma\ge k+1$, then the inverse inequality, \eqref{eq:mesh-7} and \eqref{eq:interpolation-theory} yield
\begin{equation*}
\begin{aligned}
&\vert-\sum_{j=N/2+2}^{N}\int_{I_{j}}b(x)(E-L_{k}E)\xi'\mr{d}x\vert\le C\sum_{j=N/2+2}^{N}\Vert E-L_{k}E\Vert_{L^{\infty}(I_{j})}\Vert\xi'\Vert_{L^{1}(I_{j})}\\
&\le C\sum_{j=N/2+2}^{N} h_{j}^{k+1}\Vert E^{(k+1)}\Vert_{L^{\infty}(I_{j})}h_{j}^{\frac{1}{2}}\Vert \xi'\Vert_{I_{j}}\\
&\le C\sum_{j=N/2+2}^{N} h_{j}^{k+1}e^{-\alpha(1-x_{j})/\varepsilon}\varepsilon^{-(k+1)}h_{j}^{\frac{1}{2}}\Vert \xi'\Vert_{I_{j}}\\
&\le C\sum_{j=N/2+2}^{N}N^{-(k+1)}\varepsilon^{\frac{1}{2}}\Vert \xi'\Vert_{I_{j}}\\
&\le CN^{-(k+1)}\left(\sum_{j=N/2+2}^{N}1^{2}\right)^{\frac{1}{2}}\left(\sum_{j=N/2+2}^{N}\varepsilon \Vert \xi'\Vert_{I_{j}}^{2}\right)^{\frac{1}{2}}\\
&\le CN^{-(k+\frac{1}{2})}\Vert\xi\Vert_{NIPG}.
\end{aligned}
\end{equation*}
Besides, using \eqref{eq:QQ-4} and recall that the values of $\mu(x_{j})$ \eqref{eq: penalization parameters}, 
\begin{equation*}
\begin{aligned}
&\vert-\sum_{j=N/2+2}^{N}b(x_{j})[\xi(x_{j})]\eta(x_{j}^{-})\vert\\&\le C\left(\sum_{j=N/2+2}^{N}\mu^{-1}(x_{j})\eta(x_{j}^{-})^{2}\right)^{\frac{1}{2}}\left(\sum_{j=N/2+2}^{N}\mu(x_{j})[\xi(x_{j})]^{2}\right)^{\frac{1}{2}}\\
&\le C\left(\sum_{j=N/2+2}^{N}\mu^{-1}(x_{j})\Vert\eta\Vert^{2}_{L^{\infty}(I_{j})}\right)^{\frac{1}{2}}\Vert\xi\Vert_{NIPG}\\
&\le CN^{-(k+\frac{3}{2})}\Vert\xi\Vert_{NIPG}.
\end{aligned}
\end{equation*}
Therefore, we derive
\begin{equation}\label{convergence-5}
\Gamma_{5}+\Gamma_{6}\le CN^{-(k+\frac{1}{2})}\Vert\xi\Vert_{NIPG}
\end{equation}
without any difficulties.

For $\Gamma_{7}$, from H\"{o}lder inequalities and \eqref{eq:QQ-1}, we have
\begin{equation}\label{convergence-6}
\begin{aligned}
\Gamma_{7}\le C\Vert\eta\Vert_{[0, 1]}\Vert\xi\Vert_{NIPG}\le CN^{-(k+1)}\Vert\xi\Vert_{NIPG}.
\end{aligned}
\end{equation}

Finally, according to \eqref{eq:convergence-1}, \eqref{convergence-2}, \eqref{convergence-3}, \eqref{convergence-4}, \eqref{convergence-5} and \eqref{convergence-6}, one has
\begin{equation*}
\begin{aligned}
\Vert\xi\Vert^{2}_{NIPG}&\le\Gamma_{1}+\Gamma_{2}+\Gamma_{3}+\Gamma_{4}+\Gamma_{5}+\Gamma_{6}+\Gamma_{7}\\
&\le CN^{-(k+\frac{1}{2})}\Vert\xi\Vert_{NIPG},
\end{aligned}
\end{equation*}
which implies the following estimate holds true, that is
\begin{equation}\label{eq:MA}
\Vert \Pi u-u_{N}\Vert_{NIPG}\le CN^{-(k+\frac{1}{2})}.
\end{equation}

Now we will present the main conclusion of this paper.
\begin{theorem}\label{eq:main result2}
Suppose that Assumption \ref{ass:S-1} holds true and $\mu(x_{j})$ is defined as \eqref{eq: penalization parameters}. Then on Bakhvalov-type mesh \eqref{eq:Bakhvalov-type mesh-Roos} with $\sigma\ge k+1$, we have
\begin{align*}
\Vert L_{k}u-u_{N}\Vert_{NIPG}+\Vert \Pi u-u_{N} \Vert_{NIPG}\le CN^{-(k+\frac{1}{2})},
\end{align*}
where $\Pi u$ is the interpolation defined as \eqref{eq:H-1}, and $u_N$ is the solution of \eqref{eq:SD}. 
\end{theorem}
\begin{proof}
From the triangle inequality, we have
\begin{equation*}
\begin{aligned}
\Vert L_{k}u-u_{N}\Vert_{NIPG}\le \Vert L_{k}u-\Pi u\Vert_{NIPG}+\Vert \Pi u-u_{N}\Vert_{NIPG}.
\end{aligned}
\end{equation*}
According to \eqref{eq:MA}, we just estimate the bound of $\Vert L_{k}u-\Pi u\Vert_{NIPG}$.

By the definition of $\Pi u$ \eqref{eq:H-1} and the triangle inequality, 
\begin{equation*}
\begin{aligned}
\Vert L_{k}u-\Pi u\Vert_{NIPG}&=\Vert L_{k}u-L_{k}u\Vert_{NIPG, [x_{N/2+1}, 1]}+\Vert L_{k}u-P_{h}u\Vert_{NIPG, [0, x_{N/2+1}]}\\
&=\Vert L_{k}u-P_{h}u\Vert_{NIPG, [0, x_{N/2+1}]}\\
&\le \Vert L_{k}u-u\Vert_{NIPG, [0, x_{N/2+1}]}+\Vert u-P_{h}u\Vert_{NIPG, [0, x_{N/2+1}]}.
\end{aligned}
\end{equation*}
Therefore from \eqref{post-process-2} and some direct calculations, we obtain
\begin{equation*}
\Vert L_{k}u-P_{h}u\Vert_{NIPG, [0, x_{N/2+1}]}\le CN^{-(k+\frac{1}{2})}.
\end{equation*}
Note that $\mu(x_{j})=1, j=0, 1, \cdots, N/2+1$. So far, we have completed the proof.
\end{proof} 

\begin{remark}
The reason why we choose to use Gau{\ss} Radau interpolation in $[0, x_{N/2+1}]$ is that the convergence analysis of convection term in $I_{N/2+1}$ can not analyze the past by using the standard Lagrange interpolation. In short, on $[x_{N/2}, x_{N/2+1}]$, we can't get a $\varepsilon^{\frac{1}{2}}$, making $\varepsilon^{\frac{1}{2}}\Vert\xi'\Vert_{I_{N/2+1}}\le \Vert\xi\Vert_{NIPG}$. More specifically,
\begin{equation*}
\begin{aligned}
\int_{x_{N/2}}^{x_{N/2+1}}b(x)(E-E_{I})\xi'\mr{d}x&\le C\Vert E-E_{I}\Vert_{I_{N/2+1}}\Vert\xi'\Vert_{I_{N/2+1}}\le CN^{-(\sigma+\frac{1}{2})}\Vert\xi'\Vert_{I_{N/2+1}}\\&\le C\varepsilon^{-\frac{1}{2}}N^{-(\sigma+\frac{1}{2})}\Vert\xi\Vert_{NIPG},
\end{aligned}
\end{equation*}
where $E_{I}$ represents the standard Lagrange interpolation of $E$. This difficulty can be easily handled by Gau{\ss} Radau interpolation.
\end{remark}
\begin{theorem}\label{the:main result1}
Let $\mu(x_{j})$ defined in \eqref{eq: penalization parameters} and Assumption \ref{ass:S-1} hold true. Then on Bakhvalov-type mesh \eqref{eq:Bakhvalov-type mesh-Roos} with $\sigma\ge k+1$, we have
\begin{align*}
\Vert u-u_{N} \Vert_{NIPG}\le CN^{-k},
\end{align*}
where $u$ is the exact solution of \eqref{eq:S-1}, and $u_N$ is the solution of \eqref{eq:SD}. 
\end{theorem}
\begin{proof}
Combining Theorem \ref{QQQQ-1} and \eqref{eq:MA}, we can draw this conclusion directly.
\end{proof}

\section{Numerical experiment}
In order to verify the theoretical conclusion about supercloseness, we consider the following test problem,
\begin{equation*}\label{eq:KK-2}
\left\{
\begin{aligned}
 &-\varepsilon u''(x)+(3-x)u'(x)+u(x)=f(x),\quad x\in \Omega: = (0,1),\\
&u(0)=u(1)=0.
\end{aligned}
\right.
\end{equation*}
Here $f(x)$ is chosen such that
\begin{equation*}
u(x)=x-x\cdot e^{-2(1-x)/\varepsilon}
\end{equation*}
is the exact solution of \label{eq:KK-2}.

In our numerical experiment, we first consider $\varepsilon= 10^{-5}, \cdots, 10^{-9}, k = 1, 2$ and $N =8, \cdots, 1024$. 
And on Bakhvalov-type mesh \eqref{eq:Bakhvalov-type mesh-Roos}, we set $\sigma = k+1$, $\alpha=2$ and
\begin{equation*}
\mu(x_{j})=\left\{
\begin{aligned}
&1,\quad j=0, 1, 2,\cdots, N/2,\\
&N^{2},\quad j=N/2+1, \cdots, N.
\end{aligned}
\right.
\end{equation*}
Now the corresponding convergence rate is defined as
$$r_{N}= \frac{\ln e_{N}-\ln e_{2N}}{\ln 2},$$
where for a particular $\varepsilon$, $e_{N}= \Vert L_{k}u-u_{N}\Vert_{NIPG}$ is the calculation error related to the mesh parameter $N$. 
Below, we present the following tables, which imply Theorem \ref{eq:main result2} is correct.\\
\vspace{-0.9cm}
\begin{table}[H]
\caption{$\Vert L_{k}u-u_N\Vert_{NIPG}$ for $k=1$ on Bakhvalov-type mesh}
\footnotesize
\centering
\resizebox{110mm}{25mm}{
\setlength\tabcolsep{4pt}
\begin{tabular*}{\textwidth}{@{\extracolsep{\fill}} c cccccccccccc}
\cline{1-11}
    &\multicolumn{10}{c}{$\varepsilon$ }\\
\cline{1-11}
            \multirow{2}{*}{ $N$ }   &\multicolumn{2}{c}{$10^{-5}$} &\multicolumn{2}{c}{$10^{-6}$}  &\multicolumn{2}{c}{$10^{-7}$}   
&\multicolumn{2}{c}{$10^{-8}$} &\multicolumn{2}{c}{$10^{-9}$}\\
\cline{2-11}&$e_{N}$&$r_{N}$&$e_{N}$&$r_{N}$&$e_{N}$&$r_{N}$&$e_{N}$&$r_{N}$&$e_{N}$&$r_{N}$\\
\cline{1-11}
             $8$       &0.695E-1  &1.96 &0.702E-1  &1.96 &0.706E-1  &1.96 &0.710E-1  &1.96  &0.723E-1  &1.96  \\
             $16$       &0.179E-1  &1.98  &0.180E-1  &1.98  &0.181E-1  &1.98 &0.182E-1  &1.98  &0.183E-1  &1.98  \\
             $32$       & 0.453E-2  &1.99  & 0.457E-2  &1.99 & 0.460E-2  &1.99 & 0.463E-2  &1.99 & 0.464E-2  &1.99 \\
             $64$       & 0.114E-2 &2.00  & 0.115E-2  &2.00 & 0.116E-2  &2.00  & 0.117E-2  &1.99  & 0.117E-2  &1.99  \\
             $128$      &0.285E-3  &2.00 &0.288E-3  &2.00 &0.291E-3  &2.00 &0.292E-3  &2.00 &0.294E-3  &2.00 \\
             $256$     &0.711E-4 &2.01 &0.721E-4 &2.00 &0.727E-4 &2.00 &0.732E-4 &2.00 &0.735E-4 &2.00 \\
             $512$    & 0.177E-4  & 2.01  & 0.180E-4  &2.00  & 0.182E-4  &2.00  & 0.183E-4  &2.00  & 0.184E-4 &2.00 \\
             $1024$    & 0.441E-5 &--   & 0.448E-5  &--  & 0.453E-5  &-- & 0.457E-5 &--  &0.459E-5  &--  \\

\cline{1-11}
\end{tabular*}}
\label{table:1}
\end{table}

\begin{table}[H]
\caption{$\Vert L_{k}u-u_N\Vert_{NIPG}$ for $k=2$ on Bakhvalov-type mesh}
\footnotesize
\centering
\resizebox{110mm}{25mm}{
\setlength\tabcolsep{4pt}
\begin{tabular*}{\textwidth}{@{\extracolsep{\fill}} c cccccccccc}
\cline{1-11}
    &\multicolumn{10}{c}{$\varepsilon$ }\\
\cline{1-11}
            \multirow{2}{*}{ $N$ }   &\multicolumn{2}{c}{$10^{-5}$} &\multicolumn{2}{c}{$10^{-6}$}  &\multicolumn{2}{c}{$10^{-7}$}   
&\multicolumn{2}{c}{$10^{-8}$} &\multicolumn{2}{c}{$10^{-9}$} \\

\cline{2-11}&$e_{N}$&$r_{N}$&$e_{N}$&$r_{N}$&$e_{N}$&$r_{N}$&$e_{N}$&$r_{N}$&$e_{N}$&$r_{N}$\\
\cline{1-11}
             $8$       &0.213E-1  &2.71 &0.216E-1  &2.71 &0.218E-1  &2.72  &0.220E-1  &2.72 &0.221E-1  &2.72  \\
             $16$       &0.326E-2  &2.67  &0.329E-2  &2.67 &0.332E-2  &2.68 &0.334E-2 &2.68 &0.335E-2  &2.68  \\
             $32$       & 0.512E-3  &2.62& 0.516E-3  &2.62 &0.519E-3  &2.62
&0.521E-3  &2.62  & 0.523E-3  &2.63  \\
             $64$       & 0.836E-4  &2.57  & 0.840E-4  &2.57 &0.843E-4 &2.58      &0.846E-4  &2.58  & 0.847E-4  &2.53  \\
             $128$      &0.141E-4  &2.54 &0.141E-4  &2.54&0.141E-4  &2.54 &0.142E-4  &2.48 &0.147E-4  &1.01  \\
             $256$     &0.242E-5 &2.52 &0.242E-5 &2.52 &0.242E-5 &2.47
&0.253E-5 &0.74 &0.730E-5 &-0.98 \\
             $512$      & 0.420E-6  &2.51  & 0.421E-6  &2.43  &0.438E-6 &0.63    &0.152E-5  &-0.81 & 0.144E-4  &-0.91 \\
             $1024$    & 0.737E-7 &--   & 0.781E-7  &--  &0.283E-6 &--             &0.266E-5  &-- & 0.271E-4 &--   \\

\cline{1-11}
\end{tabular*}}
\label{table:2}
\end{table}
\begin{table}[H]
\caption{$\Vert L_{k}u-u_N\Vert_{NIPG}$ for $k=3$ on Bakhvalov-type mesh}
\footnotesize
\centering
\resizebox{110mm}{25mm}{
\setlength\tabcolsep{4pt}
\begin{tabular*}{\textwidth}{@{\extracolsep{\fill}} c cccccccccc}
\cline{1-11}
    &\multicolumn{10}{c}{$\varepsilon$ }\\
\cline{1-11}
            \multirow{2}{*}{ $N$ }   &\multicolumn{2}{c}{$10^{-5}$} &\multicolumn{2}{c}{$10^{-6}$}  &\multicolumn{2}{c}{$10^{-7}$}   
&\multicolumn{2}{c}{$10^{-8}$} &\multicolumn{2}{c}{$10^{-9}$} \\

\cline{2-11}&$e_{N}$&$r_{N}$&$e_{N}$&$r_{N}$&$e_{N}$&$r_{N}$&$e_{N}$&$r_{N}$&$e_{N}$&$r_{N}$\\
\cline{1-11}

             $8$       &0.370E-2  &3.96 &0.382E-2  &3.96 &0.391E-2  &3.96 &0.397E-2  &3.96 &0.402E-2  &3.96 \\
             $16$       &0.237E-3 &3.97 &0.246E-3 &3.97  &0.252E-3 &3.96 &0.256E-3 &3.96  &0.259E-3 &3.98 \\
             $32$       &0.152E-4  &3.98 &0.157E-4  &3.97  &0.161E-4  &3.97  &0.164E-4  &3.89  &0.164E-4  &2.09 \\
             $64$       & 0.963E-6  &3.98 & 0.100E-5  &3.98  & 0.103E-5  &3.23 & 0.113E-5  &0.68  & 0.386E-5  &-0.68 \\
             $128$      &0.609E-7  &3.90 &0.637E-7  &1.92  &0.109E-6  &-0.54 &0.701E-6  &-1.12 &0.617E-5  &-1.20 \\
             $256$     &0.409E-8 &0.32 &0.169E-7 &-0.88 &0.159E-6 &-0.99 &0.152E-5 &-0.83 &0.141E-4 &-0.93 \\
             $512$      & 0.328E-8  &-1.13 & 0.310E-7  &-1.00 & 0.316E-6  &-0.94 &0.271E-5  &-1.13  & 0.270E-4 &-1.25  \\
             $1024$    & 0.717E-8 &--   & 0.619E-7  &--  & 0.609E-6  &-- & 0.594E-5 &--  &0.642E-4  &-- \\

\cline{1-11}
\end{tabular*}}
\label{table:3}
\end{table}

In addition, we present some numerical results when $\varepsilon\ge CN^{-1}$. For this purpose, we consider $\varepsilon= 10^{-1}, \cdots, 10^{-4}, k = 1, 2, 3$ and $N =8, \cdots, 1024$. And on a Bakhvalov-type mesh \eqref{eq:Bakhvalov-type mesh-Roos}, set $\sigma = k+1$, $\alpha=2$, the following tables can be obtained.
\vspace{-0.4cm}
\begin{table}[H]
\caption{$\Vert L_{k}u-u_N\Vert_{NIPG}$ for $k=1$}
\centering
\footnotesize
\resizebox{110mm}{25mm}{
\setlength\tabcolsep{4pt}
\begin{tabular*}{\textwidth}{@{\extracolsep{\fill}} c cccccccccc}
\cline{1-9}
    &\multicolumn{8}{c}{$\varepsilon$ }\\
\cline{1-9}
            \multirow{2}{*}{ $N$ }   &\multicolumn{2}{c}{$10^{-1}$} &\multicolumn{2}{c}{$10^{-2}$}  &\multicolumn{2}{c}{$10^{-3}$}   
&\multicolumn{2}{c}{$10^{-4}$}\\
\cline{2-9}&$e_{N}$&$r_{N}$&$e_{N}$&$r_{N}$&$e_{N}$&$r_{N}$&$e_{N}$&$r_{N}$\\
\cline{1-9}
             $8$      &0.632E-1  &1.91 &0.653E-1  &1.95 &0.672E-1  &1.96 &0.686E-1 &1.96 \\
             $16$       &0.168E-1  &1.93  & 0.169E-1  &1.97  &0.173E-1  &1.98 &0.176E-1  &1.98   \\
             $32$       &0.439E-2  &1.94  & 0.431E-2  & 1.97 & 0.437E-2  &1.99 & 0.446E-2  &1.99 \\
             $64$     &0.115E-2 &1.93  & 0.110E-2  &1.96& 0.110E-2  &2.00  & 0.112E-2  &2.00  \\
             $128$      &0.300E-3  &1.91 &0.282E-3  &1.95 &0.275E-3  &2.00 &0.280E-3  &2.00 \\
             $256$     &0.800E-4 &1.84 &0.730E-4 &1.95 &0.690E-4 &1.99 &0.698E-4 &2.01 \\
             $512$    &0.223E-4  &1.74  &0.189E-4  &1.96  &0.174E-4  &1.97  & 0.174E-4  &2.01 \\
             $1024$    &0.668E-5 &--   & 0.488E-5  &--  & 0.442E-5  &-- & 0.433E-5 &--   \\

\cline{1-9}
\end{tabular*}}
\label{table:3}
\end{table}
\begin{table}[H]
\caption{$\Vert L_{k}u-u_N\Vert_{NIPG}$ for $k=2$}
\footnotesize
\centering
\resizebox{110mm}{25mm}{
\setlength\tabcolsep{4pt}
\begin{tabular*}{\textwidth}{@{\extracolsep{\fill}} c cccccccc}
\cline{1-9}
    &\multicolumn{8}{c}{$\varepsilon$ }\\
\cline{1-9}
            \multirow{2}{*}{ $N$ }   &\multicolumn{2}{c}{$10^{-1}$} &\multicolumn{2}{c}{$10^{-2}$}  &\multicolumn{2}{c}{$10^{-3}$}   
&\multicolumn{2}{c}{$10^{-4}$}\\

\cline{2-9}&$e_{N}$&$r_{N}$&$e_{N}$&$r_{N}$&$e_{N}$&$r_{N}$&$e_{N}$&$r_{N}$\\
\cline{1-9}
             $8$       &0.159E-1  &2.61& 0.188E-1  &2.67 & 0.201E-1  &2.69 &0.209E-1  &2.70\\
             $16$       &0.261E-2  &2.57  &0.296E-2  &2.63 &0.311E-2  &2.65 &0.320E-2 &2.66  \\
             $32$       & 0.441E-3  &2.53&  0.479E-3  & 2.58 &0.496E-3  &2.60
& 0.506E-3  &2.61 \\
             $64$       & 0.763E-4  &2.50  & 0.800E-4  &2.55 &0.818E-4 &  2.56 &0.829E-4  &2.57 \\
             $128$      &0.135E-4  &2.48 &0.137E-4  &2.52&0.139E-4  &2.53 &0.140E-4  &2.54  \\
             $256$     &0.242E-5 &2.45&0.239E-5 &2.51 &0.240E-5 & 2.52
&0.241E-5 &2.52 \\
             $512$      & 0.443E-6  &2.40 & 0.420E-6  &2.50  & 0.418E-6 & 2.51    &0.420E-6  &2.51\\
             $1024$    &0.837E-7 &--   & 0.744E-7  &--  &0.735E-7 &--             &0.736E-7&-- \\

\cline{1-9}
\end{tabular*}}
\label{table:4}
\end{table}
\begin{table}[H]
\caption{$\Vert L_{k}u-u_N\Vert_{NIPG}$ for $k=3$}
\footnotesize
\centering
\resizebox{110mm}{25mm}{
\setlength\tabcolsep{4pt}
\begin{tabular*}{\textwidth}{@{\extracolsep{\fill}} c cccccccc}
\cline{1-9}
    &\multicolumn{8}{c}{$\varepsilon$ }\\
\cline{1-9}
            \multirow{2}{*}{ $N$ }   &\multicolumn{2}{c}{$10^{-1}$} &\multicolumn{2}{c}{$10^{-2}$}  &\multicolumn{2}{c}{$10^{-3}$}   
&\multicolumn{2}{c}{$10^{-4}$}\\

\cline{2-9}&$e_{N}$&$r_{N}$&$e_{N}$&$r_{N}$&$e_{N}$&$r_{N}$&$e_{N}$&$r_{N}$\\
\cline{1-9}
             $8$       &0.193E-2  &3.92&0.275E-2  &3.96 &0.322E-2  &3.97 & 0.352E-2  &3.97\\
             $16$       &0.128E-3  &3.93  &0.176E-3  &3.96 &0.205E-3  &3.98 & 0.225E-3&3.85  \\
             $32$       &0.839E-5  &3.92&0.113E-4  &3.97 &0.130E-4  & 3.99
&0.156E-4 &4.11\\
             $64$       & 0.553E-6  &3.89  & 0.721E-6  & 3.97 &0.818E-6 &3.99 &0.904E-6  &3.99 \\
             $128$      & 0.372E-7  &3.84 &0.462E-7  &3.95 &0.515E-7  & 3.99 &0.568E-7  &3.99  \\
             $256$     &0.260E-8 & 3.17&0.298E-8 &2.40 &0.324E-8 &3.00
& 0.357E-8 &2.60 \\
             $512$      & 0.289E-9  &-2.99 &0.564E-9  &-2.88  &0.405E-9 &-2.48 &0.589E-9 &-2.18\\
             $1024$    &0.229E-8 &--   &0.415E-8  &--  &0.225E-8 &--             & 0.266E-8&-- \\

\cline{1-9}
\end{tabular*}}
\label{table:5}
\end{table}
%Then by using the same scheme and penalty parameters on a Shishkin mesh, we have the following log-log plots.
%\begin{figure}[H]
%\centering
%\includegraphics[width=130mm, height=80mm]{Figure1.eps}
%\vspace{-0.8cm}
%\begin{center}
%\caption{$\Vert L_{k}u-u_N\Vert_{NIPG}$ in the case of $\varepsilon=10^{-6}$ on Shishkin mesh}\label{WW-2}
%\end{center}
%\end{figure}
%\begin{figure}[H]
%\vspace{-1.4cm}
%\centering
%\includegraphics[width=130mm, height=80mm]{Figure2.eps}
%\vspace{-0.8cm}
%\begin{center}
%\caption{$\Vert L_{k}u-u_N\Vert_{NIPG}$ in the case of $\varepsilon=10^{-9}$ on Shishkin mesh}\label{WW-1}
%\end{center}
%\end{figure}
%\vspace{-0.9cm}
From Table (\ref{table:1}--\ref{table:5}), we find that with the increase of $k$ and mesh parameter $N$ or the decrease of perturbation parameter $\varepsilon$, the numerical results might be unstable. This is because the changes in the above conditions may increase the condition number of the linear system, thus increasing the difficulty of solving the linear system. Therefore, for the application of high-order numerical methods, a new iterative solver should be developed to solve these ill conditioned linear systems.

%\cite{Zha1Lv2:2022-S, Zha1Lv2:2021-H}
%
%\section*{References}
%\bibliographystyle{plain}
%\bibliography{Maxiaoqi}

\end{document}